\documentclass[11pt,reqno]{amsart}

\usepackage[text={160mm,240mm},centering]{geometry}            
\usepackage{amssymb,amsmath,setspace,geometry,indentfirst,changepage,inputenc,amsthm}
\usepackage{mathrsfs,amsfonts,savesym,graphicx,bm,color}

\usepackage{graphicx}
\usepackage{amssymb}
\usepackage{dsfont}

\usepackage{lscape}

\usepackage{tikz}
\usetikzlibrary{positioning}

\usepackage{float}

\usepackage[all,2cell]{xy}    
\UseAllTwocells               

\usepackage{amsmath,amsthm}

\usepackage[mathscr]{eucal}
\usepackage[all]{xy}
\usepackage{mathrsfs}
\usepackage{hyperref}
\usepackage{color}

\newtheorem{theorem}{Theorem}[section]
\newtheorem{lemma}[theorem]{Lemma}

\newtheorem{conjecture}[theorem]{Conjecture}
\newtheorem{corollary}[theorem]{Corollary}
\newtheorem{proposition}[theorem]{Proposition}

\newtheorem{definition-lemma}[theorem]{Definition-Lemma}
\newtheorem{definition-theorem}[theorem]{Definition-Theorem}

\newtheorem*{maintheoremA}{Theorem A}
\newtheorem*{maintheoremB}{Theorem B}
\newtheorem*{maintheoremC}{Theorem C}

\theoremstyle{definition}

\newtheorem{definition}[theorem]{Definition}

\newtheorem{remark}[theorem]{Remark}

\usepackage{fancyhdr}
\allowdisplaybreaks

\title[Analytic Criteria for Detecting Bernstein-Sato Roots]
{Analytic Criteria for Detecting Bernstein-Sato Roots}

\author{Yifan Chen}
\address{
	Department of Mathematical Sciences,
	Tsinghua University,
	Beijing, 100084, P. R. China.}
\email{c-yf20@tsinghua.org.cn}

\author{Huaiqing Zuo}
\address{Department of Mathematical Sciences,
	Tsinghua University,
	Beijing, 100084, P. R. China.}
\email{hqzuo@mail.tsinghua.edu.cn}

\begin{document}
	
	\maketitle
	
	%

	\begin{abstract}
		We study roots of Bernstein-Sato polynomials of ideals. Our first main result gives a general partial differential equation criterion for detecting roots of the Bernstein-Sato polynomial of an ideal. For homogeneous case, we recast this criterion as an analytic condition, connecting the algebraic problem of locating roots with analytic properties of the defining data. As an application, we give a series of roots for the general determinantal ideal prove the strong monodromy conjecture for them.

		
	\end{abstract}

	\tableofcontents

\section{Introduction}

The Bernstein-Sato polynomial of a nonzero regular function $f$ is the
monic polynomial of least degree for which there is a differential
operator $P(s)$ satisfying
\[
P(s)f^{s+1}=b_f(s)f^s.
\]
Although defined by a differential equation, $b_f(s)$ is a subtle
invariant of the singularities of $f$. Its roots are negative rational
numbers; their exponentials describe eigenvalues of local Milnor
monodromy, and the largest root determines the log canonical threshold.
Thus the polynomial connects the theory of $\mathcal D$-modules with
the topology of Milnor fibres and birational invariants
\cite{Bernstein_original,Rationality_of_Roots_of_B-Function_Kashiwara,
	Kashiwara_Bernstein_and_eigenvalue_of_monodromy,
	Malgrange_Bernstein_eigenvalue}.
Budur, Mustaţă, and Saito extended this invariant to coherent ideals
and proved that it is independent of the chosen generators
\cite{Bernstein_Sato_polynomial_of_arbitrary_varieties}.

Our first main result gives a polynomial PDE test for an arbitrary root of the Bernstein-Sato polynomial of an ideal.

\begin{maintheoremA}[Theorem \ref{pde_thm}]
	Let $I=(f_1,\ldots,f_r)\subset\mathbb C[x_1,\ldots,x_n]$ be a nonzero
	ideal generated by nonzero polynomials with $f_j(0)=0$. Write
	$f^\beta=\prod_{j=1}^r f_j^{\beta_j}$. Suppose that for some
	$\rho\in\mathbb{Q}_{>0}$ there are polynomials
	$P_\alpha\in\mathbb C[x_1,\ldots,x_n]$, indexed by
	$\alpha\in\mathbb N^r$, such that $P_0\ne0$ and, for every
	$\alpha\in\mathbb N^r$ and $1\le i\le n$,
	\[
	x_iP_\alpha
	=
	\sum_{j=1}^r
	\left(\frac{\partial f_j}{\partial x_i}\right)(\partial)
	P_{\alpha+e_j},
	\qquad
	\sum_{j=1}^r f_j(\partial)P_{\alpha+e_j}
	=(\rho+|\alpha|)P_\alpha,
	\]
	and
	\[
	f^\beta(\partial)P_\alpha=0
	\quad\text{for every }\beta\in\mathbb N^r
	\text{ with }|\beta|=|\alpha|+1.
	\]
	Then $b_I(-\rho)=0$.
\end{maintheoremA}

In the notations above, if all $f_{i}$ are homogeneous with the same degree $d$ and $\rho$ is the distinguished root $\frac{n}{d}$, one can get the polynomial PDE in \cite[Theorem 3.9]{CZ26}. Moreover, in the homogeneous case there is a convenience judgment for the roots of the Bernstein-Sato polynomial deduced from Theorem A using conditions in analysis. This is our Theorem B.

\begin{maintheoremB}[Theorem \ref{rank_thm}]
	Let $I=(f_1,\ldots,f_r)\subset\mathbb C[x_1,\ldots,x_n]$ be generated
	by nonzero homogeneous polynomials of degree $d\ge1$, and set
	$F=(f_1,\ldots,f_r):\mathbb C^n\to\mathbb C^r$.
	If $\operatorname{rank}dF(a)=n$ for some $a\in\mathbb C^n$, then
	\[
	b_I\left(-\frac{n+q}{d}\right)=0
	\qquad\text{for every integer }0\le q<d.
	\]
\end{maintheoremB}

The monodromy conjecture is one of the most important and influential problems in singularity theory of algebraic varieties. It connects many different fields including topology, analysis, geometry and number theory. It predicts that exponentials of the poles of motivic zeta functions are eigenvalues of local monodromy. Its strong form asserts that every pole of the motivic zeta function is a root of the Bernstein-Sato polynomial. For ideals, Verdier monodromy provides the relevant monodromy action, and the
Bernstein-Sato polynomial of the ideal records its eigenvalues after
exponentiation \cite{Motivic_Igusa_Zeta_Function,
	Verdier_new_monodromy,Veys_generalized_monodromy_zeta_function,
	Bernstein_Sato_polynomial_of_arbitrary_varieties}. 
	
As an application of Theorem A and Theorem B, we give a series of roots for the Bernstein-Sato polynomial of determinantal ideals and prove the strong monodromy conjecture for determinantal varieties. Note that the (weak) monodromy conjecture for determinantal varieties was proved in \cite{CZ25} and the strong monodromy conjecture was solved for maximal minors in \cite{Bernstein_for_maximal_minors}. Moreover, the authors of \cite{Bernstein_for_maximal_minors} gave the Bernstein-Sato polynomial of the maximal minors of a generic $m \times n(m \le n)$ matrix as $\prod_{i=n-m+1}^{n}(s+i)$.

\begin{maintheoremC}[Theorem \ref{SMC_for_det}]
	Let $I_r$ be the ideal of $r\times r$ minors of a generic $m\times n$
	matrix, where $m\le n$. If $r<m$, then
	\[
	b_{I_r}\left(-\frac{(m-k)(n-k)+q}{r-k}\right)=0
	\]
	for $0\le k\le r-1$ and $0\le q\le r-k-1$. In particular, for $r \le m$ every
	pole of $Z^{\mathrm{mot}}_{I_r}(s)$ is a root of $b_{I_r}(s)$, i.e. the strong monodromy conjecture holds for determinantal varieties.
\end{maintheoremC}

This paper is organized as follows. In section \ref{sec2}, we give an introduction to the Bernstein-Sato polynomial, the motivic zeta function and the monodromy conjecture. In section \ref{sec3}, we prove our main theorems, including Theorem A, Theorem B and Theorem C.

\section{Preliminaries}\label{sec2}
\subsection{Bernstein-Sato Polynomials}
Throughout Section 2, let $X$ be a smooth complex algebraic variety over $\mathbb C$, let $\mathcal O_X$ be its structure sheaf, and let $\mathcal D_X$ be the sheaf of $\mathbb C$-linear differential operators on $X$. For every open subset $U\subseteq X$, write $\mathcal D_U=\mathcal D_X|_U$. All constructions below are local on $X$. Let us begin with the case of a single regular function. For a detailed introduction, see \cite{Popa_ref}.
\begin{definition}\label{def_b_function}
	Let $f\in\Gamma(X,\mathcal O_X)$ be a nonconstant regular function. There are \(P(s)\in\mathcal{D}_X[s]\) and \(b(s)\in\mathbb{C}[s]\) such that \(P(s)f^{s+1}=b(s)f^s\), where \(s\) is an indeterminate (see, for example, \cite{Popa_ref}). The set of all such polynomials \(b(s)\) forms an ideal in \(\mathbb{C}[s]\). Since \(\mathbb{C}[s]\) is a principal ideal domain, this ideal has a unique monic generator, which we denote by \(b_f(s)\); this is the Bernstein-Sato polynomial of \(f\).
\end{definition}

\begin{remark}
	By definition, the Bernstein-Sato polynomial of $f$ is the minimal polynomial of the action of $s$ on the sheaf quotient $\frac{\mathcal D_X[s]f^s}{\mathcal D_X[s]f^{s+1}}$.
\end{remark}

Kashiwara proved that the roots of the Bernstein--Sato polynomial are negative rational numbers.

\begin{theorem}[\cite{Rationality_of_Roots_of_B-Function_Kashiwara}]
Every root of a Bernstein-Sato polynomial is a negative rational number.
\end{theorem}

There also exists a local version of the Bernstein-Sato polynomial, with the following definition.
\begin{definition}
	Let \(x\in X\), and let \(f_x\in\mathcal{O}_{X,x}\) denote the germ of \(f\) at \(x\), where \(f\) is a nonzero regular function. The local Bernstein-Sato polynomial of $f$ at $x$, denoted by $b_{f,x}(s)$, is the unique monic polynomial of smallest degree such that
	$$
	b_{f,x}(s)f_x^s
	\in
	\mathcal D_{X,x}[s]\cdot f_x^{s+1}.
	$$
	Equivalently, there exists $P_x(s)\in\mathcal D_{X,x}[s]$
	such that $P_x(s)f_x^{s+1}=b_{f,x}(s)f_x^s$.
\end{definition}

For an open subset \(U\subseteq X\), let \(b_{f,U}(s)\) denote the Bernstein-Sato polynomial of the restriction \(f|_U\), as in Definition \ref{def_b_function}. In particular, $b_f(s)=b_{f,X}(s)$.

\begin{theorem}[Local-global relation for one function]\label{thm:local-global-function}
	Let $f\in\Gamma(X,\mathcal O_X)$ be a nonzero regular function. For every open subset $U\subseteq X$, $b_{f,U}(s)=\operatorname{lcm}_{x\in U}b_{f,x}(s)$.
	Consequently, if $V\subseteq U$ is open, then $b_{f,V}(s)\mid b_{f,U}(s)$
	and for every $x\in U$, $b_{f,x}(s)\mid b_{f,U}(s)$.
	In particular, $b_f(s)=\operatorname{lcm}_{x\in X}b_{f,x}(s)$,
	and hence $b_{f,x}(s)\mid b_f(s)$
	for every $x\in X$.
\end{theorem}

Budur, Mustaţă, and Saito introduced Bernstein-Sato polynomials for arbitrary closed subschemes (equivalently, coherent ideal sheaves) using the \(V\)-filtration; see \cite[Sections 2.4 and 2.10]{Bernstein_Sato_polynomial_of_arbitrary_varieties}. This construction extends the classical hypersurface case. By \cite[Theorem~2.5]{Bernstein_Sato_polynomial_of_arbitrary_varieties}, it is independent of the chosen generators of the ideal.

Before we give the definition of Bernstein-Sato polynomial for an ideal, let us fix some notation. Let $\mathfrak a\subseteq\mathcal O_X$ be a nonzero coherent ideal sheaf. Let $U\subseteq X$ be an open subset on which $\mathfrak a$ is generated by nonzero regular functions
$\mathfrak a|_U=(f_1,\ldots,f_r)$. Let $s_1,\ldots,s_r$ be independent variables, and put $s=s_1+\cdots+s_r$.
We regard the formal symbol \(f_1^{s_1}\cdots f_r^{s_r}\) as an element of \(\mathcal{O}_U[f_1^{-1},\ldots,f_r^{-1},s_1,\ldots,s_r]f_1^{s_1}\cdots f_r^{s_r}\). For a positive integer $m$, we use the binomial notation $\binom{s_i}{m}=\frac{s_i(s_i-1)\cdots(s_i-m+1)}{m!}$.

\begin{definition}[Bernstein-Sato polynomial of an ideal]\label{def_b_function_ideal_open_subset}
	With the notation above, the Bernstein-Sato polynomial of $\mathfrak a$ on $U$, denoted by $b_{\mathfrak a,U}(s)$, is the monic polynomial of smallest degree such that
	\begin{align}\label{def_b_function_ideal}
		b_{\mathfrak{a},U}(s)f_1^{s_1}\cdots f_r^{s_r}
		\in
		\sum_{\substack{u\in \mathbb{Z}^r\\ |u|=1}}
		\mathcal{D}_U[s_1,\ldots,s_r]\cdot
		\left(
		\prod_{u_i<0}\binom{s_i}{-u_i}
		\right)
		f_1^{s_1+u_1}\cdots f_r^{s_r+u_r},
	\end{align}
	where $|u|=u_1+\cdots+u_r$.
	Mustaţă uses this formulation in \cite[Introduction]{Mus22}. It is equivalent to the definition of Budur, Mustaţă, and Saito in \cite[Section~2.10]{Bernstein_Sato_polynomial_of_arbitrary_varieties}.
\end{definition}

In the case $r=1$, the only vector $u\in\mathbb Z$ with $|u|=1$ is $u=1$, and the definition reduces to $b_{\mathfrak a,U}(s)f^s
\in
\mathcal D_U[s]\cdot f^{s+1}$.
Thus, if $\mathfrak a|_U=(f)$, then $b_{\mathfrak a,U}(s)=b_{f,U}(s)$, where $b_{f,U}(s)$ denotes the Bernstein-Sato polynomial of the restriction of $f$ to $U$. Hence $b_{\mathfrak a,U}(s)$ extends the Bernstein-Sato polynomial of a single regular function to coherent ideal sheaves.

The role of these binomial coefficients may not be immediately clear. The \(V\)-filtration provides a useful way to understand their appearance. Let
$\mathbf f=(f_1,\ldots,f_r):U\longrightarrow\mathbb A^r$
and let $i_{\mathbf f}:U\hookrightarrow Y:=U\times\mathbb A^r$
be its graph embedding. Denote the coordinates on $\mathbb{A}^r$ by $t_1,\ldots,t_r$.
The direct image \((i_{\mathbf f})_+\mathcal{O}_U\) is a regular holonomic \(\mathcal{D}_Y\)-module and is isomorphic to \(\mathcal{D}_Y f_1^{s_1}\cdots f_r^{s_r}\), where \(\mathcal{D}_Y\) denotes the sheaf of differential operators on \(Y\). 
Consider the ambient module $\mathcal N_{\mathbf{f}}=\mathcal O_U[(f_{1}...f_{r})^{-1},s_{1},...,s_{r}]f_{1}^{s_{1}}...f_{r}^{s_{r}}$,
with the following actions:
$$t_{i}(h(x,s_{1},...,s_{r})f_{1}^{s_{1}}...f_{r}^{s_{r}})=f_{i}\,h(x,s_{1},...,s_{i}+1,...,s_{r})f_{1}^{s_{1}}...f_{r}^{s_{r}}$$ and

$$\partial_{t_{i}}(h(x,s_{1},...,s_{r})f_{1}^{s_{1}}...f_{r}^{s_{r}})=-s_{i}f_{i}^{-1}\,h(x,s_{1},...,s_{i}-1,...,s_{r})f_{1}^{s_{1}}...f_{r}^{s_{r}}.$$
The graph module $(i_\mathbf{f})_+\mathcal O_U$ is identified with
the cyclic submodule $\mathcal D_{Y}f_{1}^{s_{1}}...f_{r}^{s_{r}}\subseteq\mathcal N_{\mathbf{f}}$.
Expressions involving $t_{i}^{-1}$ are understood
in the ambient module. Let $s_{ij}:=-\partial_{t_{i}}t_{j}$. Equivalently, \(b_{\mathfrak a,U}(s)\) is the minimal polynomial of the action of \(s=\sum_{i=1}^{r}-\partial_{t_i}t_i\) on
$$\frac{\mathcal{D}_{U}[s_{i},s_{ij}]f_{1}^{s_{1}}...f_{r}^{s_{r}}}{\sum_{k=1}^{r}\mathcal{D}_{U}[s_{i},s_{ij}]f_{k}\cdot f_{1}^{s_{1}}...f_{r}^{s_{r}}}.$$
\begin{remark}\label{uniqueness}
The module $\mathcal{D}_{Y}f_{1}^{s_{1}}...f_{r}^{s_{r}}$ has many good properties. One of them is that its elements can be written uniquely as a sum of terms $h_{\alpha_{1},...,\alpha_{r}}(x)\partial_{t_{1}}^{\alpha_{1}}...\partial_{t_{r}}^{\alpha_{r}}f_{1}^{s_{1}}...f_{r}^{s_{r}}$, where $h_{\alpha_{1},...,\alpha_{r}}(x) \in \mathcal{O}_{U}$. In fact,
by the graph-module action, the elements
$\partial_{t_{1}}^{\alpha_{1}}...\partial_{t_{r}}^{\alpha_{r}}f_{1}^{s_{1}}...f_{r}^{s_{r}}$ span $\mathcal{D}_{Y}f_{1}^{s_{1}}...f_{r}^{s_{r}}$ over
$\mathcal{O}_{U}$. In the ambient module $\mathcal N_{\mathbf{f}}=\mathcal O_U[(f_{1}...f_{r})^{-1},s_{1},...,s_{r}]f_{1}^{s_{1}}...f_{r}^{s_{r}}$
they satisfy
\[
\partial_{t_{1}}^{\alpha_{1}}...\partial_{t_{r}}^{\alpha_{r}}f_{1}^{s_{1}}...f_{r}^{s_{r}}
=
(-1)^{|\alpha|}
\prod_{j=1}^r
\left((s_j)_{\alpha_j}f_j^{-\alpha_j}\right)f^s,
\qquad
(s_j)_k:=s_j(s_j-1)\cdots(s_j-k+1),
\]
with $(s_j)_0=1$. Since the products
$\prod_j(s_j)_{\alpha_j}$ form a basis of
$\mathbb C[s_1,\ldots,s_r]$, these elements are linearly independent
over $\mathcal{O}_{U}$. Hence every element of $\mathcal{D}_{Y}f_{1}^{s_{1}}...f_{r}^{s_{r}}$
has a unique finite expression  as a sum of terms $h_{\alpha_{1},...,\alpha_{r}}(x)\partial_{t_{1}}^{\alpha_{1}}...\partial_{t_{r}}^{\alpha_{r}}f_{1}^{s_{1}}...f_{r}^{s_{r}}$. 
\end{remark}

Budur, Mustaţă, and Saito proved that \(b_{\mathfrak a,U}(s)\) exists and is independent of the choice of generators; see \cite[Theorem~2.5]{Bernstein_Sato_polynomial_of_arbitrary_varieties}. 
\begin{theorem}[Existence and independence of the choice of generators]
	Let $\mathfrak a\subseteq\mathcal O_X$ be a nonzero coherent ideal sheaf. On an open subset $U\subseteq X$, suppose that
	$\mathfrak a|_U=(f_1,\ldots,f_r)=(g_1,\ldots,g_m)$.
	Then $b_{\mathfrak a,U}(s)$ exists, and the two systems of generators define the same polynomial.
\end{theorem}

As for a single regular function, one can also define a local Bernstein-Sato polynomial for an ideal.
\begin{definition}
	Let $\mathfrak a\subseteq\mathcal O_X$ be a nonzero coherent ideal sheaf and let $x\in X$. Choose generators $\mathfrak a_x=(f_1,\ldots,f_r)$ in $\mathcal O_{X,x}$, and put $s=s_1+\cdots+s_r$.
	The local Bernstein-Sato polynomial of $\mathfrak a$ at $x$, denoted by $b_{\mathfrak a,x}(s)$, is the unique monic polynomial of smallest degree such that $b_{\mathfrak a,x}(s_1+\cdots+s_r)
	f_1^{s_1}\cdots f_r^{s_r}$
	belongs to
	$
	\sum_{\substack{u\in\mathbb Z^r\\ |u|=1}}
	\mathcal D_{X,x}[s_1,\ldots,s_r]
	\left(
	\prod_{u_i<0}\binom{s_i}{-u_i}
	\right)
	f_1^{s_1+u_1}\cdots f_r^{s_r+u_r}$.
	By \cite[Theorem 2.5]{Bernstein_Sato_polynomial_of_arbitrary_varieties}, this polynomial is independent of the chosen generators of $\mathfrak a_x$.
\end{definition}

\begin{definition}
	For every open subset $U\subseteq X$, define
	$$
	b_{\mathfrak a,U}(s)
	:=
	\operatorname{lcm}_{x\in U}b_{\mathfrak a,x}(s).
	$$
	By \cite[Section~2.10, formula~(2.10.1)]{Bernstein_Sato_polynomial_of_arbitrary_varieties}, this definition agrees with Definition~\ref{def_b_function_ideal_open_subset} whenever \(\mathfrak a|_U\) is generated by regular functions on \(U\). The global Bernstein-Sato polynomial of $\mathfrak a$ is defined by
	$$
	b_{\mathfrak a}(s)
	:=
	b_{\mathfrak a,X}(s)
	=
	\operatorname{lcm}_{x\in X}b_{\mathfrak a,x}(s).
	$$
\end{definition}

\begin{theorem}[Local-global relation for ideals]\label{thm:local-global-ideal}
	Let $\mathfrak a\subseteq\mathcal O_X$ be a nonzero coherent ideal sheaf. If $V\subseteq U\subseteq X$ are open subsets, then $b_{\mathfrak a|_V,V}(s)
	\mid
	b_{\mathfrak a|_U,U}(s)$.
	Moreover, for every $x\in U$,
	$b_{\mathfrak a,x}(s)
	\mid
	b_{\mathfrak a|_U,U}(s)$.
	In particular, $b_{\mathfrak a,x}(s)\mid b_{\mathfrak a}(s)$
	for every $x\in X$.
\end{theorem}

\subsection{Motivic Zeta Functions and the Monodromy Conjecture}\label{monodromy_conj}

Originally formulated for a single regular function, the monodromy conjecture predicts a relation among the poles of the motivic zeta function, the eigenvalues of local Milnor monodromy, and the roots of Bernstein-Sato polynomials. Having introduced Bernstein-Sato polynomials in the preceding subsection, we now recall the relevant notions of Milnor monodromy and motivic zeta functions.

Let $f\in\Gamma(X,\mathcal O_X)$ be a nonconstant regular function. For each \(x\in f^{-1}(0)\), the analytic map defined by \(f\) has a local Milnor fiber at \(x\), whose cohomology carries a monodromy action. Let \(E(f)\) denote the set of eigenvalues occurring in the cohomology of the local Milnor fibers at points \(x\in V(f)\). These eigenvalues are closely related to the roots of the Bernstein-Sato polynomial of \(f\). Kashiwara \cite{Kashiwara_Bernstein_and_eigenvalue_of_monodromy} and Malgrange \cite{Malgrange_Bernstein_eigenvalue} proved that \(\mathrm{Exp}(V(b_f)):=\{e^{2\pi\sqrt{-1}s_0}\mid b_f(s_0)=0\}\) equals \(E(f)\).

The conjecture also concerns the motivic zeta function. The motivic zeta function is defined by motivic integration and can be computed from a log resolution, as follows. The theorem also applies to ideals; its original form is recovered by taking \(I=(f)\).

\begin{theorem}[\cite{Motivic_Igusa_Zeta_Function}]\label{theorem_denef_loeser}
	Let $X$ be a smooth complex algebraic variety, let $\mathfrak a\subseteq\mathcal O_X$ be a nonzero coherent ideal sheaf, and let
	$\mu:Y\longrightarrow X$ be a log resolution of $\mathfrak a$. Let \(E_i\) (\(i\in S\)) be the irreducible components appearing in either \(\operatorname{div}(\mu^*(\mathfrak a))\) or \(K_{Y/X}\), and write \(\operatorname{div}(\mu^*(\mathfrak a))=\sum_{i\in S}N_iE_i\) and \(K_{Y/X}=\sum_{i\in S}(\nu_i-1)E_i\). Then the motivic zeta function of $\mathfrak{a}$, denoted by $Z^{\mathrm{mot}}_{\mathfrak{a}}(s)$, is given by
	\begin{flalign*}
		Z^{\mathrm{mot}}_{\mathfrak{a}}(s)=\sum_{J \subset S}[\mathring{E}_{J}]\prod_{i \in J}\frac{\mathbb{L}-1}{\mathbb{L}^{N_{i}s+\nu_{i}}-1},
	\end{flalign*}
	where $\mathring{E}_{J}:=(\cap_{i \in J}E_{i})\setminus (\cup_{j \notin J}E_{j})$. We refer to the pairs \((N_i,\nu_i)\), \(i\in S\), as the numerical data of the log resolution.
\end{theorem} 

The topological zeta function is obtained by specializing the motivic zeta function via the Euler characteristic, as follows.
\begin{definition}[Topological zeta function]
	As in the setting above, the topological zeta function of $\mathfrak{a}$ is given by
	\begin{flalign*}
		Z^{\mathrm{top}}_{\mathfrak{a}}(s):=\sum_{J \subset S}\chi(\mathring{E}_{J})\prod_{i \in J}\frac{1}{N_{i}s+\nu_{i}}.
	\end{flalign*}
\end{definition}

Now we can state the (weak) monodromy conjecture and the strong monodromy conjecture.

\begin{conjecture}[Monodromy Conjecture, \cite{Motivic_Igusa_Zeta_Function}]\label{mc}
	Let $X$ be a smooth complex algebraic variety and let $f\in\Gamma(X,\mathcal O_X)$ be a nonconstant regular function. If $s_0$ is a pole of $Z_{f}^{\mathrm{mot}}(s)$, then $e^{s_0\cdot 2\pi \sqrt{-1}} \in E(f)$.
\end{conjecture}

\begin{conjecture}[Strong Monodromy Conjecture]\label{smc}
	Let $X$ be a smooth complex algebraic variety and let $f\in\Gamma(X,\mathcal O_X)$ be a nonconstant regular function. If $s_0$ is a pole of $Z_{f}^{\mathrm{mot}}(s)$, then $s_{0}$ is a root of $b_{f}(s)$.
\end{conjecture}

The monodromy conjecture also has an ideal-theoretic version. In this version, one replaces \(f\) by \(\mathfrak a\) and \(Z_f^{\mathrm{mot}}\) by \(Z_{\mathfrak a}^{\mathrm{mot}}\). The Bernstein-Sato polynomial also has an ideal-theoretic version, denoted by \(b_{\mathfrak a}(s)\); see \cite{Bernstein_Sato_polynomial_of_arbitrary_varieties}. In the ideal setting, there is no direct analogue of the local Milnor monodromy eigenvalues. Verdier introduced Verdier monodromy in \cite{Verdier_new_monodromy}; in the hypersurface setting, it agrees with Milnor monodromy and extends naturally to ideals. This ideal-theoretic version of monodromy is compatible with the Bernstein-Sato polynomial of \(\mathfrak a\). Budur, Mustaţă, and Saito proved in \cite{Bernstein_Sato_polynomial_of_arbitrary_varieties} that \(\mathrm{Exp}(V(b_{\mathfrak a}(s))):=\{e^{2\pi\sqrt{-1}s_0}\mid b_{\mathfrak a}(s_0)=0\}\) is precisely the set of eigenvalues of Verdier monodromy.

	\section{Criteria for the roots of the Bernstein-Sato polynomial}\label{sec3}
	In this section, we will focus on the affine case and prove our main theorems. We will fix $X=\mathbb{C}^{n}$ and $\mathcal{D}=\mathbb{C}[x_{1},...,x_{n},\partial_{1},...,\partial_{n}]$. First we will prove a polynomial PDE criterion for the general roots of Bernstein-Sato polynomial, which is a generalization of Theorem A in \cite{CZ26}. The proof will use Fischer pairing, which we define as follows.
	
	\begin{definition}
		The Fischer pairing is a bilinear pairing on the polynomial ring such that
		\begin{flalign*}
			\langle \cdot, \cdot \rangle: \mathbb{C}[x_{1},...,x_{n}] \times \mathbb{C}[x_{1},...,x_{n}] &\rightarrow \mathbb{C}  \\
			(f(x),g(x)) &\mapsto f(\partial) \cdot g(x)|_{x=0},
		\end{flalign*}
		where $f(\partial)$ denotes the differential operator obtained by replacing every $x_{i}$ with the corresponding $\partial_{i}$. 
	\end{definition}
	Since $\mathbb{C}[x_{1},...,x_{n}]$ is a $\mathbb{C}$-vector space generated by all monomials, it suffices to see how this pairing acts on monomials. By definition, we have
	\begin{flalign*}
		\langle x^{\alpha},x^{\beta} \rangle=\begin{cases}
			\alpha!   \ \mathrm{if} \ \alpha=\beta,   \\
			0   \  \mathrm{if} \ \alpha \neq \beta,
		\end{cases}
	\end{flalign*}
	where $\alpha=(\alpha_{1},...,\alpha_{n}),\beta=(\beta_{1},...,\beta_{n})$ are vectors, $x^{\alpha}=x_{1}^{\alpha_{1}}...x_{n}^{\alpha_{n}}$ and $\alpha !=\alpha_{1}!\alpha_{2}!...\alpha_{n}!$.
	
	We give some properties of the Fischer pairing within the following lemma. They can be proved easily using the characterization of the pairing above.
	
	\begin{lemma}\label{Fischer_pairing}
		The Fischer pairing has the following properties.
		
		$(1)$ This pairing is commutative, $\mathbb{C}$-bilinear and non-degenerate.
		
		$(2)$ The adjoint of some polynomial can be translated into the action of a differential operator, that is, for $f,g,h \in \mathbb{C}[x_{1},...,x_{n}]$, we have $\langle h(x)f(x),g(x) \rangle = \langle f(x), h(\partial) \cdot g(x)\rangle$.
		
	\end{lemma}
	\begin{theorem}\label{pde_thm}
		Let $I=(f_{1},...,f_{r}) \subset \mathbb{C}[x_{1},...,x_{n}]$ be an ideal with $f_{i}(0)=0$. If there exist $P_{\alpha} \in \mathbb{C}[x_{1},...,x_{n}]$ for $\alpha=(\alpha_{1},...,\alpha_{r}) \in \mathbb{N}^{r}$ and $\rho \in \mathbb{Q}_{>0}$ satisfying the following conditions.
		
		$(1)$ $P_{0} \neq 0$;
		
		$(2)$ $x_{i}P_{\alpha}=\sum_{j=1}^{r}\partial_{i}f_{j}(\partial) \cdot P_{\alpha+e_{j}}$;
		
		$(3)$ $\sum_{j=1}^{r}f_{j}(\partial) \cdot P_{\alpha+e_{j}}=(\rho+|\alpha|)P_{\alpha}$;
		
		$(4)$ For any $\beta=(\beta_{1},...,\beta_{r}) \in \mathbb{N}^{r}$ such that $|\beta|=|\alpha|+1$, we have $f_{1}^{\beta_{1}}...f_{r}^{\beta_{r}}(\partial) \cdot P_{\alpha}=0$.
		
		Then $b_{I}(-\rho)=0$.
	\end{theorem}
	\begin{proof}
		Consider the module $\mathcal{M}=\mathcal{D}_{x,t}f_{1}^{s_{1}}...f_{r}^{s_{r}}$. Its element can be written as the linear combination of the form $h\partial_{t}^{\alpha}f^{s}$, where $h \in \mathbb{C}[x_{1},...,x_{n}]$. We define a linear map
		\begin{flalign*}
		\Lambda:\mathcal{M} &\rightarrow \mathbb{C}  \\
		h \partial_{t}^{\alpha}f^{s} &\mapsto \langle P_{\alpha}, h \rangle.	
		\end{flalign*}
		By Remark \ref{uniqueness}, this $\Lambda$ is well-defined.
		For any $h\partial_{t}^{\alpha}f^{s} \in \mathcal{M}$ and any $1 \le i \le n$, we can calculate $\Lambda(\partial_{i}(h\partial_{t}^{\alpha}f^{s}))$ as follows.
		\begin{flalign*}
			&\Lambda(\partial_{i}(h\partial_{t}^{\alpha}f^{s}))  \\
		   =&\Lambda(\partial_{i}h\partial_{t}^{\alpha}f^{s})-\Lambda(h\sum_{j=1}^{r}\partial_{i}f_{j}\partial_{t_{j}}(\partial_{t}^{\alpha}f^{s}))  \\
		   =&\langle P_{\alpha}, \partial_{i}h \rangle-\sum_{j=1}^{r}\langle P_{\alpha+e_{j}}, h\partial_{i}f_{j} \rangle \\
		   =&\langle  x_{i}P_{\alpha}-\sum_{j=1}^{r}\partial_{i}f_{j}(\partial) \cdot P_{\alpha+e_{j}},h \rangle \\
		   =&0.
		\end{flalign*}
		This implies that for any $m \in \mathcal{M}$, we have 
		\begin{align}\label{equation_1}
			\Lambda(\partial_{i}m)=0
		\end{align}
		
		Now we come to the identity of b function, that is, there exist $P_{j} \in \mathcal{D}[s_{i},s_{ij}]$ such that 
		\begin{flalign*}
			b_{I}(s)f^{s}=\sum_{j=1}^{r}P_{j}f_{j}f^{s}.
		\end{flalign*}
		Since $P_{0} \neq 0$, one can find a polynomial $h_{0} \neq 0$ such that $\langle P_{0},h_{0} \rangle \neq 0$. We will prove that $\Lambda(\sum_{j=1}^{r}h_{0}P_{j}f_{j}f^{s})=0$. To do this, we write every term of $h_{0}P_{j}$ as $\partial_{x}^{\gamma}x^{\beta}\partial_{t}^{\theta}t^{\alpha}$ such that $|\theta|=|\alpha|$. Using equation \ref{equation_1}, we only need to consider the term with $\gamma=0$. Now we have
		\begin{flalign*}
			&\Lambda(x^{\beta}\partial_{t}^{\theta}t^{\alpha}f_{j}f^{s}) \\
			=&\Lambda(\partial_{t}^{\theta}x^{\beta}f^{\alpha}f_{j}f^{s}) \\
			=&\langle P_{\theta}, x^{\beta}f^{\alpha}f_{j} \rangle \\
			=&\langle f^{\alpha}f_{j}(\partial) \cdot P_{\theta},x^{\beta} \rangle  \\
			=&0.
		\end{flalign*}
		This tells us $\Lambda(b_{I}(s)h_{0}f^{s})=0$. On the other hand, for any $h\partial_{t}^{\alpha}f^{s} \in \mathcal{M}$, we have
		\begin{flalign*}
			&\Lambda(h\partial_{t}^{\alpha}sf^{s}) \\
			=&\Lambda(-\sum_{j=1}^{r}h\partial_{t}^{\alpha}\partial_{t_{j}}t_{j}f^{s}) \\
			=&-\sum_{j=1}^{r}\Lambda(hf_{j}\partial_{t}^{\alpha+e_{j}}f^{s}) \\
			=&-\sum_{j=1}^{r}\langle P_{\alpha+e_{j}},hf_{j} \rangle \\
			=&-\langle \sum_{j=1}^{r}f_{j}(\partial) \cdot P_{\alpha+e_{j}},h \rangle  \\
			=&-\langle (\rho+|\alpha|)P_{\alpha},h \rangle \\
			=&-(\rho+|\alpha|)\Lambda(h\partial_{t}^{\alpha}f^{s}).
		\end{flalign*}
		Meanwhile, 
		\begin{flalign*}
			\partial_{t}^{\alpha}s=s\partial_{t}^{\alpha}-|\alpha|\partial_{t}^{\alpha}.
		\end{flalign*}
		This implies that $\Lambda(sm)=-\rho\Lambda(m)$ for any $m \in \mathcal{M}$. Applying $\Lambda$ to the equation $b_{I}(s)h_{0}f^{s}=\sum_{j=1}^{r}h_{0}P_{j}f_{j}f^{s}$ we can get
		\begin{flalign*}
			0=&\Lambda(b_{I}(s)h_{0}f^{s})  \\
			 =&b_{I}(-\rho)\langle P_{0},h_{0} \rangle.
		\end{flalign*}
		We have $b_{I}(-\rho)=0$.
	\end{proof}

	In the homogeneous and weighted homogeneous cases, we have the following corollaries.
	\begin{corollary}
	Let $I=(f_{1},...,f_{r}) \subset \mathbb{C}[x_{1},...,x_{n}]$ be an ideal generated by weighted homogeneous polynomials $f_{i}$ with common weights. We assume $\mathrm{wt}(x_{i})=w_{i} \in \mathbb{N}$ and $\mathrm{wt}(f_{i})=d$. Suppose there exist $P_{\alpha} \in \mathbb{C}[x_{1},...,x_{n}]$ indexed by $\alpha \in \mathbb{N}^{r}$ which are weighted homogeneous polynomials with weight $d|\alpha|+q$ for some integer $0 \le q<d$ satisfying the following conditions.
	
	$(1)$ $P_{0} \neq 0$;
	
	$(2)$  $x_{i}P_{\alpha}=\sum_{j=1}^{r}\partial_{i}f_{j}(\partial) \cdot P_{\alpha+e_{j}}$,
	
	where $e_{j}=(0,...,0,1,0,...,0)$ is the $j$-th unit vector. Then $b_{I}(-\frac{q+\sum_{i=1}^{n}w_{i}}{d})=0$.
	\end{corollary}
	\begin{proof}
	It suffices to verify the conditions $(3)$ and $(4)$ in Theorem \ref{pde_thm}. Since $f_{j}$ are weighted homogeneous with weight $d$, we have the Euler identity
	\begin{flalign*}
		\sum_{i=1}^{n}w_{i}x_{i}\partial_{i}f_{j}=d \cdot f_{j}.
	\end{flalign*}
	By replacing every $x_{i}$ with the corresponding $\partial_{i}$, taking it into the condition $(3)$ in Theorem \ref{pde_thm} and using condition $(2)$, one can get
	\begin{flalign*}
		 &\sum_{j=1}^{r}f_{j}(\partial) \cdot P_{\alpha+e_{j}}  \\
		=&d^{-1}\sum_{j=1}^{r}\sum_{i=1}^{n}w_{i}\partial_{i} \cdot \partial_{i}f_{j}(\partial) \cdot P_{\alpha+e_{j}}  \\
		=&d^{-1}\sum_{i=1}^{n}w_{i}\partial_{i}x_{i}P_{\alpha}  \\
		=&d^{-1}\sum_{i=1}^{n}w_{i}P_{\alpha}+d^{-1}\sum_{i=1}^{n}w_{i}x_{i}\partial_{i}P_{\alpha}  \\
		=&(|\alpha|+\frac{q+\sum_{i=1}^{n}w_{i}}{d})P_{\alpha}.
	\end{flalign*}
	Thus the condition $(3)$ in Theorem \ref{pde_thm} holds for $\rho=\frac{q+\sum_{i=1}^{n}w_{i}}{d}$. For condition (4), give each $\partial_i$ weight $-w_i$. Since each
	$f_j(\partial)$ is weighted homogeneous of weight $-d$, the operator $\prod_{j=1}^r f_j(\partial)^{\beta_j}$ lowers the weight by $d|\beta|$. If $|\beta|=|\alpha|+1$, then any nonzero result of applying this operator to $P_\alpha$ would have weight $d|\alpha|+q-d|\beta|=q-d<0$.
	The polynomial ring has no nonzero elements of negative weight, so $\prod_{j=1}^r f_j(\partial)^{\beta_j}P_\alpha=0$. Hence condition (4) holds. In conclusion, Theorem \ref{pde_thm} implies $b_I(-\rho)=0$, as claimed. 
	\end{proof}
	
	\begin{corollary}\label{homogeneous_pde}
		Let $I=(f_{1},...,f_{r}) \subset \mathbb{C}[x_{1},...,x_{n}]$ be an ideal generated by homogeneous polynomials $f_{i}$ of degree $d$.  Suppose there exist $P_{\alpha} \in \mathbb{C}[x_{1},...,x_{n}]$ indexed by $\alpha \in \mathbb{N}^{r}$ which are homogeneous polynomials with degree $d|\alpha|+q$ for some integer $0 \le q<d$ satisfying the following conditions.
		
		$(1)$ $P_{0} \neq 0$;
		
		$(2)$  $x_{i}P_{\alpha}=\sum_{j=1}^{r}\partial_{i}f_{j}(\partial) \cdot P_{\alpha+e_{j}}$,
		
		where $e_{j}=(0,...,1,...,0)$ is the $j$-th unit vector. Then $b_{I}(-\frac{q+n}{d})=0$.
	\end{corollary}

	\begin{theorem}\label{rank_thm}
	Suppose $I=(f_{1},...,f_{r}) \subset \mathbb{C}[x_{1},...,x_{n}]$ is an ideal generated by homogeneous polynomials of degree $d$. If there exists $a \in \mathbb{C}^{n}$ such that for $F:=(f_{1},...,f_{r}):\mathbb{C}^{n} \rightarrow \mathbb{C}^{r}$, the rank of $dF(a)$ is $n$. Then $b_{I}(-\frac{n+q}{d})=0$ for $0 \le q <d$. 
	\end{theorem}
	\begin{proof}
	First we prove the case when $r=n$. Since the Jacobian matrix of $F$ at $a$ is invertible, by holomorphic inverse function theorem, there exists an open neighborhood $U$ of $a$ such that $F|_{U}:U \rightarrow V:=F(U)$ is biholomorphic. Let $\phi:V \rightarrow U$ be its inverse. We denote the coordinate on $U$ by $x$ and on $V$ by $v$ and we define $D(x):=\mathrm{det}(dF(x))$. We prove the following identity which can be obtained by changing variables for the divergence formula.
		\begin{flalign*}
			\frac{\partial_{i}h(\phi(v))}{D(\phi(v))}=\sum_{j}\partial_{v_{j}}\frac{\partial_{i}f_{j}(\phi(v))h(\phi(v))}{D(\phi(v))}.
		\end{flalign*}
		More specifically, let $W=h(x)\partial_{x_{i}}$ be a vector field. The divergence of $W$ with respect to $\omega:=dx_{1}\wedge...\wedge dx_{n}$ is $\partial_{i}h(x)$. When we use the coordinate $v$, we have $W=h(\phi(v))\sum_{j}\frac{\partial f_{j}}{\partial x_{i}} \partial_{v_{j}}$ and $\omega=(D(\phi(v)))^{-1}dv_{1} \wedge ...\wedge dv_{n}$. By definition of divergence, we have 
		\begin{flalign*}
			d(\iota_{W}\omega)=\mathrm{div}_{\omega}W \cdot \omega.
		\end{flalign*}
		Expanding this equation in the coordinate $v$, the right hand side is $\partial_{i}h(\phi(v))(D(\phi(v)))^{-1}dv_{1} \wedge ...\wedge dv_{n}$. After computing, the left hand side is $\sum_{j=1}^{n}\partial_{v_{j}}\frac{\partial_{i}f_{j}(\phi(v))h(\phi(v))}{D(\phi(v))}dv_{1} \wedge... \wedge dv_{n}$. Comparing two sides gives the required equation.
		
		Now we define $L_{\alpha}(h)=\partial_{v}^{\alpha}(\frac{h(\phi(v))}{D(\phi(v))})|_{v=F(a)}$. Then we have
		\begin{flalign*}
			L_{\alpha}(\partial_{i}h)=&\partial_{v}^{\alpha}(\frac{\partial_{i}h(\phi(v))}{D(\phi(v))})|_{v=F(a)}  \\
			=&\sum_{j}\partial_{v}^{\alpha+e_{j}}\frac{\partial_{i}f_{j}(\phi(v))h(\phi(v))}{D(\phi(v))}|_{v=F(a)}  \\
			=&\sum_{j}L_{\alpha+e_{j}}(\partial_{i}f_{j}h).
		\end{flalign*}
		Fix $q<d$, and we define $\lambda_{\alpha}(h):=L_{\alpha}(h_{d|\alpha|+q})$, where $h_{d|\alpha|+q}$ denotes the degree $d|\alpha|+q$ part of $h$. Since $f_{i}$ is homogeneous of degree $d$ for all $i$, the equation above also holds for $\lambda_{\alpha}$, that is,
		\begin{align}\label{equation_2}
		\lambda_{\alpha}(\partial_{i}h)=\sum_{j}\lambda_{\alpha+e_{j}}(\partial_{i}f_{j}h).
		\end{align}
		We take the dual of $\lambda_{\alpha}$ under Fischer pairing. More explicitly, we define $$P_{\alpha}:=\sum_{|\beta|=d|\alpha|+q}\frac{\lambda_{\alpha}(x^{\beta})}{\beta!}x^{\beta}.$$ Then $P_{\alpha}$ is homogeneous of degree $d|\alpha|+q$ and $\langle P_{\alpha},h \rangle=\lambda_{\alpha}(h)$ for any $h$. By Fischer pairing and equation \ref{equation_2}, we have $x_{i}P_{\alpha}=\sum_{k=1}^{r}\partial_{i}f_{k}P_{\alpha+e_{k}}$. 
		It remains to verify that $P_0\ne0$. Since $D(a)\ne0$, for every
		homogeneous polynomial $h$ of degree $q$ we have $\lambda_0(h)=\frac{h(a)}{D(a)}$. If $q=0$, take $h=1$. If $q>0$, then $d\ge2$, so $dF(0)=0$ and the rank assumption implies $a\ne0$. Choose a linear form $\ell$ such that
		$\ell(a)\ne0$. Then $\lambda_0(\ell^q)=\frac{\ell(a)^q}{D(a)}\ne0$.
		Thus $\lambda_0$ is nonzero, and its Fischer-pairing dual $P_0$ is
		nonzero. The $P_{\alpha}$ are the required polynomials for Corollary \ref{homogeneous_pde}, so the theorem holds for $r=n$.
		
		Next we deal with the general case. Since the rank of $dF(a)$ is $n$, we have $r \ge n$. Without loss of generality we may assume the Jacobian matrix of $(f_{1},...,f_{n})$ at $a$ is invertible. By the proof above, we can find homogeneous polynomials $Q_{\beta}$ indexed by $\beta \in \mathbb{N}^{n}$ such that they satisfy the conditions in Corollary \ref{homogeneous_pde}. For $\alpha=(\alpha_{1},...,\alpha_{r}) \in \mathbb{N}^{r}$, we construct the required polynomials $P_{\alpha}$ in Corollary \ref{homogeneous_pde} as follows. If any of $\alpha_{n+1},...,\alpha_{r}$ is not zero, we let $P_{\alpha}=0$ and if $\alpha_{n+1}=...=\alpha_{r}=0$, we let $P_{\alpha}=Q_{\alpha_{1},...,\alpha_{n}}$. One can verify that these $P_{\alpha}$ satisfy the conditions in Corollary \ref{homogeneous_pde}, so the result holds.
	\end{proof}
	
	We will use the theorem to prove a series of roots of the Bernstein-Sato polynomial of the general determinantal ideal. Before that, we need to fix some notations. Let $\mathcal{M}=(x_{ij})_{m \times n}$ be a $m \times n$ matrix with variables $x_{ij}(1 \le i \le m, 1 \le j \le n)$. For convenience we assume $m \le n$. For integer $1 \le r \le m$, let $I_{r}$ be the ideal generated by all $r$-th minors of $\mathcal{M}$. To use Theorem \ref{rank_thm}, we suppose $I_{r}=(f_{1},...,f_{k})$ and $F:=(f_{1},...,f_{k}): \mathbb{C}^{mn} \rightarrow \mathbb{C}^{k}$, where $f_{1},...,f_{k}$ are all $r \times r$ minors of $\mathcal{M}$. We fix a matrix $A=(I_{m} \ 0)$ and we prove the following lemma.
	
	\begin{lemma}\label{lemma_rank}
	With the notations above, if $r<m$, then the rank of $dF(A)$ is $mn$.
	\end{lemma}
	\begin{proof}
	The proof is straightforward. Let $H=(h_{ij})$ be a $m \times n$ matrix in the kernel of $dF(A)$, we want to prove that $H=0$. For subset $I \subset \{1,...,m\}$ and $J \subset \{1,...,n\}$ with $|I|=|J|$, we denote the minor generated by rows and columns indexed by $I$ and $J$ by $\delta_{IJ}$. Since $H \in \mathrm{ker}(dF(A))$, we have $\sum_{ij}(\partial_{x_{ij}}\delta_{IJ})(A) \cdot h_{ij}=0$ for all $r \times r$ minors $\delta_{IJ}$. For $h_{pq}$ with $p \neq q$, since $r<m$, there exists a subset $K \subset \{1,...,m\}$ such that $|K|=r-1$ and $p,q \notin K$. We take $I=K \cup \{p\}$ and $J=K \cup \{q\}$ and by considering the equation $\sum_{ij}(\partial_{x_{ij}}\delta_{IJ})(A) \cdot h_{ij}=0$, we can get $h_{pq}=0$. Now we consider the equation $\sum_{ij}(\partial_{x_{ij}}\delta_{IJ})(A) \cdot h_{ij}=0$ for case when $I=J$. Then we have $\sum_{i \in I}h_{ii}=0$. Since this summation is over any subset $I \subset \{1,...,m\}$ and $r<m$, all $h_{ii}$ are equal, thus $h_{ii}=0$ for all $i$. In conclusion we have $H=0$, and the result follows.
	\end{proof}
	
	Now we can give a series of roots of the Bernstein-Sato polynomial of the determinantal ideal $I_{r}$ and prove the strong monodromy conjecture of it. Note that the strong monodromy conjecture for maximal minors has been proved in \cite{Bernstein_for_maximal_minors}.
	
	\begin{theorem}[\cite{Bernstein_for_maximal_minors}]
	In the notations above, the strong monodromy conjecture for $I_{m}$ holds.
	\end{theorem}
	
	\begin{theorem}\label{SMC_for_det}
	For integer $1\le r <m$, $0 \le k \le r-1$ and $0 \le q \le r-k-1$, we have $b_{I_{r}}(-\frac{(m-k)(n-k)+q}{r-k})=0$. In particular, the strong monodromy conjecture for $I_{r}$ holds for $r<m$. Combining with \cite{Bernstein_for_maximal_minors} we know that the strong monodromy conjecture holds for general determinantal varieties. 
	\end{theorem}
	\begin{proof}
	 By Theorem \ref{rank_thm} and Lemma \ref{lemma_rank}, we have $b_{I_{r}}(-\frac{mn+q}{r})=0$ for $0 \le q \le r-1$. Next we consider the small determinantal variety defined by $(r-k) \times (r-k)$ minors in a $(m-k) \times (n-k)$ matrix for $1 \le k \le r-1$. By \cite[Remark 4.3]{CZ26}, the roots of the Bernstein-Sato polynomials for these varieties are also the roots of $b_{I_{r}}$. We can repeat the argument of Theorem \ref{rank_thm} and Lemma \ref{lemma_rank} to conclude that $-\frac{(m-k)(n-k)+q}{r-k}$ are also roots of $b_{I_{r}}$ for $0 \le k \le r-1$ and $0 \le q \le r-k-1$. By \cite[Remark 3.4]{CZ25}, the poles of the motivic zeta function of $I_{r}$ are $-\frac{(m-k)(n-k)}{r-k}$, so the strong monodromy conjecture for $I_{r}$ holds.
	\end{proof}
	\begin{remark}
	Note that Theorem \ref{SMC_for_det} only give part of the roots of the Bernstein-Sato polynomials of determinantal varieties, while \cite[Theorem B]{CZ25} gives all the possible eigenvalues of the Verdier monodromy. Although these eigenvalues can be covered by the roots given in Theorem \ref{SMC_for_det}, \cite[Theorem B]{CZ25} can not be covered by Theorem \ref{SMC_for_det} and we still do not know the Bernstein-Sato polynomials of general determinantal varieties.
	\end{remark}
	
	Another application of the Theorem \ref{pde_thm} is a cohomological judgment of the roots of the Bernstein-Sato polynomial. Note that this judgment can be covered by \cite{Ta26}, but it is discovered on our own.
	\begin{proposition}
		Let $f \in \mathbb{C}[x_{1},...,x_{n}]$ be a polynomial with $f(0)=0$ and let $\pi:Y \rightarrow \mathbb{C}^{n}$ be a log resolution of $V(f)$. Assume $g=f \circ \pi$ and $E_{i}(i \in S)$ are the irreducible components of $\mathrm{div}(g)$ and $K_{Y/\mathbb{C}^{n}}$ such that $\mathrm{div}(g)=\sum_{i \in S}N_{i}E_{i}$ and $K_{Y/\mathbb{C}^{n}}=\sum_{i \in S}(\nu_{i}-1)E_{i}$. Fix some $E_{0}$ such that $\pi(E_{0})=\{0\}$ and let $N=N_{0}, \nu=\nu_{0}$. Let $\rho=\frac{\nu}{N}$ and $U:=E_{0}\setminus \cup_{j \neq 0}E_{j}$. We define a multi-valued differential form $\omega_{\rho}$ on $E_{0}$ as follows.
		$$\omega_{\rho}:=\frac{g^{1-\frac{\nu}{N}}\pi^{*}(dx_{1} \wedge ... \wedge dx_{n})}{dg}|_{E_{0}}.$$
		The quotient in this formula is understood as the residue along
		$E_0$. More explicitly, near any point of $U$, choose local coordinates
		$(y,w_1,\ldots,w_{n-1})$ with $E_0=\{y=0\}$ and write
		\[
		g=u\,y^N,\qquad
		\pi^*(dx_1\wedge\cdots\wedge dx_n)
		=v\,y^{\nu-1}\,dy\wedge dw_1\wedge\cdots\wedge dw_{n-1},
		\]
		where $u$ and $v$ are nowhere-vanishing holomorphic functions. On a
		chosen branch, this gives
		\[
		\omega_\rho
		=\frac{1}{N}u(0,w)^{-\rho}v(0,w)\,
		dw_1\wedge\cdots\wedge dw_{n-1}.
		\]
		Here $H^{n-1}(U,\mathcal{L}_\rho)$ is understood as the cohomology of
		the de Rham complex with coefficients in the rank-one local
		system determined by these branches. If $[\omega_{\rho}] \in H^{n-1}(U,\mathcal{L}_{\rho})$ is non-zero, then $b_{f}(-\rho)=0$.
	\end{proposition}
	\begin{proof}
		Let $\lambda:\mathbb{C} \rightarrow \mathbb{C}, t \mapsto t^{N}$ and $p_{1}: Y \times_{\mathbb{C}} \mathbb{C}$ be the fiber product. Suppose $p_{0}:Y' \rightarrow Y \times_{\mathbb{C}} \mathbb{C}$ is the normalization of $Y \times_{\mathbb{C}} \mathbb{C}$ and $p:=p_{1} \circ p_{0}: Y' \rightarrow Y$ is the composition. Let $\Phi:=\pi \circ p$ and we have $f \circ \Phi=g \circ p$. 
		
		After shrinking each coordinate neighborhood, choose an $N$-th root
		$h$ of $u$. Near every point of $U$, the base change $t^N=g$ is given
		by
		\[
		t^N=u y^N,\qquad
		\prod_{\xi^N=1}(t-\xi h y)=0.
		\]
		Its normalization is the disjoint union of the $N$ smooth branches
		$t=\zeta h y$, and $p$ maps each branch biholomorphically onto the
		chosen neighborhood. Choose an open neighborhood $W$ of $U$ in $Y$
		such that
		\[
		W\cap\bigcup_{j\ne 0}E_j=\varnothing,
		\]
		and put $W'=p^{-1}(W)$ and $U'=p^{-1}(U)$. At points of
		$W\setminus E_0$, the function $g$ is nonzero, so the base change
		$t^N=g$ is unramified there as well. Since $p$ is finite, these local
		descriptions show that $p|_{W'}:W'\to W$ is a finite unramified
		covering of degree $N$. In particular, $p|_{U'}:U'\to U$ is the restriction of this cyclic covering.
		
		Let $\tau$ be the pullback to $Y'$ of the coordinate $t$. The local expression $\tau=\xi h y$ shows that $d\tau\ne 0$ along $U'$. Thus, near every point of $U'$, we can choose holomorphic coordinates $(\tau,w_1,\ldots,w_{n-1})$ on $W'$ such that $U'=\{\tau=0\}$ and $g\circ p=\tau^N$. Since the order of $K_{Y/\mathbb{C}^{n}}$ over $E_{0}$ is $\nu-1$, we have $\Phi^{*}(dx_{1} \wedge...\wedge dx_{n})=N\tau^{\nu-1}d\tau \wedge \Theta$ and $\Theta=N^{-1}a(\tau,w_{1},...,w_{n-1})dw_{1}\wedge ... \wedge dw_{n-1}$. Let $\Theta_{t}:=\Theta|_{\tau=t}$, and through direct computation one can verify that $p^{*}\omega_{\rho}=\Theta_{0}$. 
	
		Moreover, for $\xi=e^{2\pi i/N}$, the deck transformation
		$D_\xi$ satisfies $D_\xi^*\Theta_0=\xi^{-\nu}\Theta_0$, $p \circ D_{\xi}=p$ and $\tau \circ D_{\xi}=\xi \tau$. We use the convention that $\mathcal{L}_\rho$ is the local system associated with this character of the cyclic covering.
		
		We claim that $[\Theta_0]\ne 0$ in $H_{\mathrm{dR}}^{n-1}(U',\mathbb{C})$.
		Suppose that $\Theta_0=d\beta$. Define
		\[
		\overline{\beta}
		:=\frac{1}{N}\sum_{k=0}^{N-1}\xi^{\nu k}(D_\xi^k)^*\beta.
		\]
		Since $D_\xi^*\Theta_0=\xi^{-\nu}\Theta_0$, we have
		$d\overline{\beta}=\Theta_0$ and
		$D_\xi^*\overline{\beta}=\xi^{-\nu}\overline{\beta}$.
		Thus $\overline{\beta}$ descends to an $\mathcal{L}_\rho$-valued
		$(n-2)$-form $\beta_\rho$ on $U$. Since $p^*\omega_\rho=\Theta_0$,
		the descended form satisfies $d\beta_\rho=\omega_\rho$,
		contradicting $[\omega_\rho]\ne0$. Hence $[\Theta_0]\ne0$.
		
		By the de Rham pairing, there is a finite smooth singular
		$(n-1)$-cycle $\Gamma$ in $U'$ such that
		$\int_\Gamma\Theta_0\ne0$. Rescaling $\Gamma$, we may assume that
		$\int_\Gamma\Theta_0=1$. Let $K$ be the support of $\Gamma$, and choose
		a relatively compact open neighborhood $V$ of $K$ in $U'$. 
		
		Let $K=\operatorname{supp}\Gamma$. Choose a compact neighborhood
		$C$ of $K$ in the fiber $U'=\tau^{-1}(0)$, and put
		$V=\operatorname{int}(C)$. By the differential tubular-neighborhood
		theorem for submersions \cite[Lemmas 3--4]{SS71}, after shrinking
		$\epsilon>0$ there is a smooth embedding
		\[
		T:\Delta_\epsilon\times V\longrightarrow W'
		\]
		such that
		\[
		T(0,p)=p,\qquad \tau(T(t,p))=t.
		\]
		
		For small $\epsilon$, we define $A_{\epsilon}:=T_{*}(\{|t|=\epsilon\} \times \Gamma)$ and 
		$$c_{k}(a):=\frac{1}{2\pi iN}\int_{A_{\epsilon}}\tau^{-\nu-Nk}\Phi^{*}(adx_{1} \wedge ... \wedge x_{n})$$
		for any $k \in \mathbb{N}$ and $a \in \mathbb{C}[x_{1},...,x_{n}]$. 
		
		This definition is independent of the choice of $\epsilon$. In fact, for $\epsilon_{1}<\epsilon_{2}$, let $A_{\epsilon_{1},\epsilon_{2}}=T_{*}(\{\epsilon_{1}<|t|<\epsilon_{2}\} \times \Gamma)$. We have $\partial A_{\epsilon_{1},\epsilon_{2}}=A_{\epsilon_{1}}-A_{\epsilon_{2}}$. By stokes theorem, 
		\begin{flalign*}
			&\frac{1}{2\pi iN}\int_{A_{\epsilon_{1}}}\tau^{-\nu-Nk}\Phi^{*}(adx_{1} \wedge ... \wedge x_{n})-\frac{1}{2\pi iN}\int_{A_{\epsilon_{2}}}\tau^{-\nu-Nk}\Phi^{*}(adx_{1} \wedge ... \wedge x_{n})   \\
			=&\frac{1}{2\pi iN}\int_{\partial A_{\epsilon_{1},\epsilon_{2}}}\tau^{-\nu-Nk}\Phi^{*}(adx_{1} \wedge ... \wedge x_{n})  \\
			=&\frac{1}{2\pi iN}\int_{A_{\epsilon_{1},\epsilon_{2}}}d(\tau^{-\nu-Nk}\Phi^{*}(adx_{1} \wedge ... \wedge x_{n}))  \\
			=&0.
		\end{flalign*}
		
		On the region $\varepsilon_1\leq|\tau|\leq\varepsilon_2$, the form
		$\tau^{-\nu-Nk}\Phi^*(a\,dx_1\wedge\cdots\wedge dx_n)$ is a holomorphic
		$n$-form. Hence it is closed, which justifies the last equality in the
		Stokes computation.
	
		We will prove several properties of $c_{k}(a)$ and construct the solution of the PDE in Theorem \ref{pde_thm} by taking dual of these $c_{k}$. Let $\Gamma_{t}:=T(t,\cdot)_{*}\Gamma$, and we have
		$$c_{k}(a)=\frac{1}{2\pi i}\int_{|t|=\epsilon}t^{-Nk-1}\int_{\Gamma_{t}}(a \circ \Phi)\Theta_{t}dt.$$
		
		Put $\eta=(a\circ\Phi)\Theta$ and
		$\eta_t=\eta|_{\tau=t}$. Since $W'_t:=\{\tau=t\}$ has complex
		dimension $n-1$, the restriction $\eta_t$ is a holomorphic
		top-degree form on $W'_t$, and hence
		$d_{W'_t}\eta_t=0$. Write $T_t(p)=T(t,p)$. Then
		\[
		F_a(t):=\int_{\Gamma_t}\eta_t
		=\int_\Gamma T_t^*\eta_t .
		\]
		Differentiating this integral with respect to $\bar t$, let
		$Z_{\bar t}$ denote the corresponding variation vector field along
		$\Gamma_t$. Since $\tau\circ T(t,p)=t$, the vector field $Z_{\bar t}$
		is tangent to the fiber $W'_t$. Cartan's formula gives
		\[
		\frac{\partial}{\partial\bar t}F_a(t)
		=\int_{\Gamma_t}\mathcal{L}_{Z_{\bar t}}\eta_t
		=\int_{\Gamma_t}d_{W'_t}
		\bigl(\iota_{Z_{\bar t}}\eta_t\bigr)
		=0,
		\]
		where we used $d_{W'_t}\eta_t=0$ and $\partial\Gamma_t=0$.
		Therefore $F_a(t)$ is holomorphic about $t$.
		
		We prove the following several properties.
		
		$(1)$ $c_{0}(a)=a(0)$.
		
		Since $\Phi|_{U'}=0$ and $\int_\Gamma\Theta_0=1$, we obtain
		\[
		c_0(a)=F_a(0)
		=\int_\Gamma(a\circ\Phi)\Theta_0
		=a(0)\int_\Gamma\Theta_0
		=a(0).
		\]
		
		$(2)$ $c_{k+1}(fa)=c_{k}(a)$.
		
		Using $f \circ \Phi=\tau^{N}$, we have $\Phi^{*}(fadx_{1} \wedge... \wedge dx_{n})=\tau^{n}\Phi^{*}(adx_{1} \wedge ... \wedge dx_{n})$, and the result follows.
		
		$(3)$ $c_{k}(\partial_{i}a)=(\rho+k)c_{k+1}(\partial_{i}f \cdot a)$.
		
		Since $f \circ \Phi=\tau^{N}$, we have $\Phi^{*}df=N\tau^{N-1}d\tau$. Fix $i$ and wedge $(-1)^{i-1}\Phi^{*}(adx_{1}\wedge...\wedge \hat{dx_{i}}\wedge...\wedge dx_{n})$, we can get
		$$\Phi^{*}(\partial_{i}f a dx_{1} \wedge ... \wedge dx_{n})=(-1)^{i-1}N\tau^{N-1}d\tau \wedge\Phi^{*}(adx_{1}\wedge...\wedge \hat{dx_{i}}\wedge...\wedge dx_{n}),$$
		where $\hat{dx_{i}}$ means $dx_{i}$ is not there. Now we can calculate as follows.
		\begin{flalign*}
			&c_{k}(\partial_{i}a)-(\rho+k)c_{k+1}(\partial_{i}f \cdot a) \\
			=&\frac{1}{2\pi iN}\int_{A_{\epsilon}}\tau^{-\nu-Nk}\Phi^{*}(\partial_{i}adx_{1} \wedge ... \wedge x_{n})-(\rho+k)\tau^{-\nu-Nk-N}\Phi^{*}(\partial_{i}f \cdot adx_{1} \wedge ... \wedge x_{n})  \\
			=&\frac{1}{2\pi iN}\int_{A_{\epsilon}}d((-1)^{i-1}\tau^{-\nu-Nk}\Phi^{*}(adx_{1} \wedge ... \wedge \hat{dx_{i}} \wedge ... \wedge dx_{n}))  \\
			=&0.
		\end{flalign*}
		
		$(4)$ $c_k$ depends only on the terms of $a$ of degree
		at most $Nk$.
		
		It suffices to show that if $a \in (x_{1},...,x_{n})^{Nk+1}$, then $c_{k}(a)=0$. Since $\Phi|_{U'}=0$, there exist holomorphic functions $b_{i}$ for every $i$ such that $x_{i} \circ \Phi=\tau b_{i}$. This implies that $\tau^{-\nu-Nk}\Phi^{*}(adx_{1} \wedge ...\wedge dx_{n})=a \circ \Phi N \tau^{-Nk-1}d\tau \wedge \Theta$ is holomorphic because the order of $\tau$ of $a \circ \Phi$ is at least $Nk+1$. Thus $c_{k}(a)=0$.
		
		Now we define a $\mathbb{C}$-linear map $\lambda_{k}:\mathbb{C}[x_{1},...,x_{n}] \rightarrow \mathbb{C}$ as $\lambda_{k}=\rho(\rho+1)...(\rho+k-1)c_{k}$. Then we have $\lambda_{0}(a)=a(0)$, $\lambda_{k+1}(fa)=(\rho+k)\lambda_{k}(a)$ and $\lambda_{k+1}(\partial_{i}f \cdot a)=\lambda_{k}(\partial_{i}a)$. 
		
		The property $(4)$ also holds for $\lambda_k$. Therefore $\lambda_k$
		is supported on the finite-dimensional space of polynomials of degree
		at most $Nk$, and its Fischer-pairing dual is the polynomial
		\[
		P_k=\sum_{|\beta|\le Nk}\frac{\lambda_k(x^\beta)}{\beta!}x^\beta,
		\qquad
		\langle P_k,a\rangle=\lambda_k(a).
		\]
		
		We now verify the conditions of Theorem \ref{pde_thm} for the polynomials $P_k$.
		First, $\lambda_0(a)=a(0)$ implies $P_0=1$. For every polynomial $a$,
		the recurrences for $\lambda_k$ and the adjointness of the Fischer
		pairing give
		\[
		\begin{aligned}
			\langle x_iP_k,a\rangle
			&=\lambda_k(\partial_i a)
			=\lambda_{k+1}((\partial_i f)a)
			=\langle (\partial_i f)(\partial)P_{k+1},a\rangle,\\
			\langle f(\partial)P_{k+1},a\rangle
			&=\lambda_{k+1}(fa)
			=(\rho+k)\lambda_k(a)
			=\langle(\rho+k)P_k,a\rangle.
		\end{aligned}
		\]
		Thus conditions (2) and (3) of Theorem 3.3 hold. For condition (4),
		using $\lambda_{j+1}(fa)=(\rho+j)\lambda_j(a)$ repeatedly, we obtain
		\[
		\begin{aligned}
			\langle f(\partial)^{k+1}P_k,a\rangle
			&=\lambda_k(f^{k+1}a)\\
			&=\left(\prod_{j=0}^{k-1}(\rho+j)\right)\lambda_0(fa)
			=0,
		\end{aligned}
		\]
		because $\lambda_0(fa)=f(0)a(0)=0$; the product is understood as $1$
		when $k=0$. Hence $f(\partial)^{k+1}P_k=0$. All four conditions of
		Theorem \ref{pde_thm} are satisfied for the principal ideal $(f)$, and therefore
		$b_f(-\rho)=0$.
	\end{proof}

	\end{document}